\documentclass{amsart}
\numberwithin{equation}{section}

  \usepackage{verbatim}
\usepackage{mathrsfs}
\usepackage{latexsym}
\usepackage{epsfig}
\usepackage{amsmath}
\usepackage{bbm}
\usepackage{graphics}
\usepackage{amssymb}
\usepackage{amsthm}
\usepackage[alphabetic]{amsrefs}
\usepackage{amsopn}
\usepackage{amscd}
\usepackage[all,knot]{xy}
\xyoption{all}
\usepackage{rotating}
\usepackage{hyperref}

\theoremstyle{plain}
\newtheorem{thm}{Theorem}[section]
\newtheorem{lem}[thm]{Lemma}
\newtheorem{prop}[thm]{Proposition}
\newtheorem{cor}[thm]{Corollary}

\newtheorem*{thm*}{Theorem}
\newtheorem*{lem*}{Lemma}
\newtheorem*{prop*}{Proposition}
\newtheorem*{cor*}{Corollary}

\theoremstyle{definition}
\newtheorem{defn}[thm]{Definition}
\newtheorem*{defn*}{Definition}

{}
\newtheorem{rem}[thm]{Remark}
\newtheorem*{rem*}{Remark}

\newtheorem{notation}[thm]{Notation}{}
{}
{}
\newtheorem*{ack}{Acknowledgements}{}

\theoremstyle{remark}
{}
{}
{}

\def\cie{\subseteq}

\def\iso{\cong}

\def\un{\cup}

\def\intersec{\cap}

\def\to{\longrightarrow}

\def\rimp{\Rightarrow}

\def\int{\mathbb{Z}}

\def\str{\mathcal{O}}

\def\cN{\mathcal{N}}

\def\s{\sigma}
\def\S{\Sigma}
\def\t{\tau}
\def\D{\mathsf{D}}
\def\K{\mathsf{K}}

\def\op{\mathrm{op}}

\DeclareMathOperator{\Spec}{Spec}
\DeclareMathOperator{\Spech}{Spec^h}
\DeclareMathOperator{\Sing}{Sing}

\DeclareMathOperator{\supp}{supp}
\DeclareMathOperator{\supph}{supph}
\DeclareMathOperator{\codim}{codim}

\DeclareMathOperator{\rk}{rk}
\DeclareMathOperator{\id}{id}

\DeclareMathOperator{\Hom}{Hom}
\DeclareMathOperator{\RHom}{\mathbf{R}Hom}

\DeclareMathOperator{\RsHom}{\mathbf{R}\mathcal{H}\mathit{om}}
\DeclareMathOperator{\Ext}{Ext}

\DeclareMathOperator{\modu}{mod}

\DeclareMathOperator{\Coh}{Coh}
\DeclareMathOperator{\QCoh}{QCoh}

\DeclareMathOperator{\MCM}{MCM}
\DeclareMathOperator{\sMCM}{\underline{\MCM}}
\DeclareMathOperator{\Inj}{Inj}

\DeclareMathOperator{\Th}{Th}

\title{Duality for bounded derived categories of complete intersections}
\author{Greg Stevenson}
\address{Universit\"at Bielefeld, Fakult\"at f\"ur Mathematik, BIREP Gruppe, Postfach 10\,01\,31, 33501 Bielefeld, Germany.}
\email{gstevens@math.uni-bielefeld.de}

\begin{document}



\begin{abstract}
\noindent We show that every thick subcategory of the singularity category of a complete intersection ring is self dual. We also prove the analogous statement for thick subcategories of the bounded derived category and give applications to the symmetry of vanishing of cohomology. These results are also proved for certain complete intersection schemes.
\end{abstract}

\maketitle

\tableofcontents

\section{Introduction}

In \cite{DTradius} Dao and Takahashi pose the following question: given a local complete intersection $R$ is it true that every resolving subcategory of the category of maximal Cohen-Macaulay $R$-modules, $\MCM(R)$, is closed under $R$-duals? We give an affirmative answer to this question using the classification of thick subcategories of $\sMCM(R)$, the stable category of maximal Cohen-Macaulay $R$-modules, which was obtained in \cite{Stevensonclass}.

It turns out that the natural way to view this question seems to be in terms of the duality $\D^b(\Coh X) \iso \D^b(\Coh X)^\op$ for a Gorenstein scheme $X$. In this case $\str_X$ is a dualising complex so the perfect complexes are stable under this duality. Hence Grothendieck duality descends to the singularity category $\D_\mathrm{Sg}(X)$. One can then ask which ``nice'' subcategories of $\D^b(\Coh X)$ and $\D_\mathrm{Sg}(X)$ are self-dual. We do not know the answer to this question in general but we show that for certain complete intersection schemes every nice thick subcategory is self-dual in both of these categories; this is proved in Theorem \ref{thm_duality} for the singularity category and in Corollary \ref{cor_duality} for the bounded derived category. This implies the orthogonality relation for nice thick subcategories of these categories is symmetric.

Along the way (Theorem \ref{thm_bij}) we give, for such schemes, a classification of thick subcategories of $\D^b(\Coh X)$ which are closed under tensoring with perfect complexes (what we mean by nice above), which extends the known classifications for the perfect complexes and the singularity category. 



\section{Duality and singular supports on complete intersections}\label{sec_duality1}

We begin by fixing some notation. For a scheme $X$ (all our schemes are noetherian, separated and assumed to have enough locally free sheaves) we denote, as is usual, the bounded derived category of coherent $\str_X$-modules by $\D^b(\Coh X)$ (or $\D^b(\modu R)$ when $X = \Spec R$) and the thick subcategory of perfect complexes by $\D^\mathrm{perf}(X)$. The \emph{singularity category} of $X$, denoted $\D_{\mathrm{Sg}}(X)$, is defined to be the (idempotent completion of the) Verdier quotient $\D^b(\Coh X)/\D^\mathrm{perf}(X)$.

For us a complete intersection ring is a ring $R$ such that there is a regular ring $Q$ and a surjection $Q\to R$ with kernel generated by a regular sequence.

Let $T$ be a separated regular noetherian scheme of finite Krull dimension and $\mathcal{E}$ a vector bundle on $T$ of rank $c$. For a section $t\in H^0(T,\mathcal{E})$ we denote by $Z(t)$ the \emph{zero scheme} of $t$. We recall $Z(t)$ can be defined globally by the exact sequence
\begin{displaymath}
\xymatrix{
\mathcal{E}^{\vee} \ar[r]^{t^{\vee}} & \str_T \ar[r] & \str_{Z(t)} \ar[r] & 0.
}
\end{displaymath}
It can also be defined locally by taking a cover $X= \un_iU_i$ trivializing $\mathcal{E}$ via $f_i\colon \mathcal{E}\vert_{U_i} \stackrel{\sim}{\to} \str_{U_i}^{\oplus c}$ and defining an ideal sheaf $\mathscr{I}(s)$ by
\begin{displaymath}
\mathscr{I}(t)\vert_{U_i} = (f_i(t)_1,\ldots, f_i(t)_c).
\end{displaymath}
We say the section $t$ is \emph{regular} if the ideal sheaf $\mathscr{I}(t)$ is locally generated by a regular sequence. Thus the zero scheme $Z(t)$ of a regular section $t$ is a locally complete intersection in $T$ of codimension $c$. In our situation $t$ is  regular if and only if $\codim Z(t) = \rk \mathcal{E} = c$ (see for example \cite{MatsuAlg}*{16.B}). 

Let $T$ and $\mathcal{E}$ be as above and let $t\in H^0(T,\mathcal{E})$ be a regular section with zero scheme $X$. Denote by $\mathcal{N}_{X/T}$ the normal bundle of $X$ in $T$. Consider the projective bundles $\mathbb{P}(\mathcal{E}^\vee) = T'$ and $\mathbb{P}(\mathcal{N}_{X/T}^\vee) = Z$ with projections $q$ and $p$ respectively. Associated to these projective bundles are canonical line bundles $\str_{\mathcal{E}}(1)$ and $\str_\mathcal{N}(1)$ together with canonical surjections
\begin{displaymath}
q^*\mathcal{E} \to \str_\mathcal{E}(1) \quad \text{and} \quad p^*\mathcal{N}_{X/S} \to \str_\mathcal{N}(1).
\end{displaymath}
The section $t$ induces a section $t'\in H^0(T', \str_\mathcal{E}(1))$ and we denote its divisor of zeroes by $Y$. The natural closed immersion $Z\to T'$ factors via $Y$. In summary we are concerned with the following commutative diagram
\begin{equation}\label{eq_diag}
\xymatrix{
Z = \mathbb{P}(\mathcal{N}_{X/T}^\vee) \ar[r]^(0.7)i \ar[d]_p & Y \ar[r]^(0.3)u  & \mathbb{P}(\mathcal{E}^\vee) = T' \ar[d]^{q} \\
X \ar[rr]_j && T
}
\end{equation}
and we use the notation occurring in it throughout.

By a theorem of Orlov \cite{OrlovSing2}*{Section 2} there is an equivalence of triangulated categories $\D_\mathrm{Sg}(X) \iso \D_\mathrm{Sg}(Y)$. We will assume throughout that $\str_{\mathcal{E}}(1)$ is ample (this is the case for instance if $T$ is affine, so all our results hold for complete intersection rings). In this case \cite{Stevensonclass}*{Theorem~8.8} tells us that thick subcategories of $\D_\mathrm{Sg}(X)$ are in order preserving bijection with specialization closed subsets of $\Sing Y$, the singular locus of $Y$.

We wish to show that every thick subcategory of $\D_\mathrm{Sg}(X)$ is closed under Grothendieck duality that is, if $\mathcal{C} \cie \D_\mathrm{Sg}(X)$ is a thick subcategory and $M\in \mathcal{C}$ then the image of $\RsHom(M,\str_X)$ is in $\mathcal{C}$. In particular we have $\mathcal{C}\iso \mathcal{C}^\op$. We do this in two steps: we first prove the functors giving Orlov's equivalence are sufficiently compatible with Grothendieck duality, and then we use this compatibility and a computation on $Y$ to prove the statement.

\begin{notation}
Suppose $Z$ is a Gorenstein scheme and $A$ is an object of $\D^b(\Coh Z)$. We use $A^\vee$ to denote $\RsHom(A, \str_Z)$ and note that the canonical morphism $A \to (A^{\vee})^{\vee}$ is an isomorphism as $\str_Z$ is a dualizing sheaf on $Z$.
\end{notation}

In the next two lemmas we describe the interaction of the functors $p^*$ and $i_*$ with the duality; these statements are surely well known to the experts. The behaviour of the pullback is standard and so we simply indicate the idea of a proof.

\begin{lem}\label{lem_pull}
Let $M$ be an object of $\D^b(\Coh X)$. Then the canonical morphism 
\begin{displaymath}
p^*M^\vee {\to} (p^*M)^\vee
\end{displaymath}
is an isomorphism.
\end{lem}
\begin{proof}
The claimed morphism always exists by nonsense. We can check locally that it is an isomorphism which reduces the proof to a standard fact from commutative algebra.
\end{proof}

Slightly more care is required for the pushforward along $i$ - we cannot expect that it commutes with duality on the nose. However, the result we need is again totally standard in our setting. Let us prepare a little before stating it. We know $i_*\colon \D^b(\Coh Z) \to \D^b(\Coh Y)$ has a right adjoint $i^!$. 
As $i$ is a regular closed immersion we know from \cite{HartshorneRD}*{III.7.3} that 
\begin{displaymath}
i^!\str_Y \iso \S^{-c+1}\wedge^{c-1}\cN_{Z/Y}
\end{displaymath}
where $\cN_{Z/Y}$ is the normal bundle. In particular $i^!\str_Y$ is a shift of a line bundle.

\begin{lem}\label{lem_push}
Let $M$ be an object of $\D^b(\Coh Z)$. Then there is a canonical isomorphism
\begin{displaymath}
(i_*M)^\vee \iso i_*(M^\vee \otimes i^!\str_Y).
\end{displaymath}
\end{lem}
\begin{proof}
The proof comes down to composing a sequence of well known isomorphisms. We have
\begin{align*}
(i_*M)^\vee &= \RsHom_Y(i_*M, \str_Y) \\
&\iso i_*\RsHom_Z(M, i^!\str_Y) \\
&\iso i_*(\RsHom_Z(M,\str_Z) \otimes i^!\str_Y) \\
&= i_*(M^\vee \otimes i^!\str_Y),
\end{align*}
where the first non-trivial isomorphism is the usual duality isomorphism and the second holds since $i^!\str_Y$ is a shift of a line bundle.
\end{proof}

Now let us briefly recall some notation and facts about singular support. We will then be ready to do the simple computation in $\D^b(\Coh Y)$ which will essentially complete the proof of the theorem. We denote the projection $\D^b(\Coh Y) \to \D_{\mathrm{Sg}}(Y)$ by $\pi$ (in fact we use this notation for the projection regardless of the scheme involved, but this should cause no confusion). We recall from \cite{Stevensonclass}*{Section 3} that there is an action of $\D(\QCoh Y)$, the unbounded derived category of $Y$, on $\K_\mathrm{ac}(\Inj Y)$ the homotopy category of acyclic complexes of injective quasi-coherent $\str_Y$-modules. This gives, as in \cite{StevensonActions}, a notion of support for objects of $\K_\mathrm{ac}(\Inj Y)$ and hence for the compact objects $\D_\mathrm{Sg}(Y) \iso \K_\mathrm{ac}(\Inj Y)^c$ (where this equivalence is due to Krause \cite{KrStab}). Rather than assuming any knowledge of this machinery let us just say that the support of the image of an object $M$ of $\D^b(\Coh Y)$ can be defined as
\begin{displaymath}
\supp \pi M = \{y\in \Sing Y\; \vert \; M_y \notin \D^\mathrm{perf}(\str_{Y,y}) \}
\end{displaymath}
(this description can be found in the proof of \cite{Stevensonclass}*{Lemma~5.12}).

\begin{lem}\label{lem_dualsupp}
Let $M$ be an object of $\D^b(\Coh Y)$. There is an equality
\begin{displaymath}
\supp \pi M = \supp \pi (M^\vee).
\end{displaymath}
\end{lem}
\begin{proof}
The result is immediate from the definition of the support we have given: for $y\in Y$ an object $N$ of $\D^b(\modu \str_{Y,y})$ is not perfect if and only if $N^\vee$ is not perfect.

\end{proof}

Given an object $M\in \D^b(\Coh X)$ we set
\begin{displaymath}
\supp_{\D(\QCoh Y)} \pi M = \supp (\pi i_*p^*M).
\end{displaymath}
We note that this is precisely the support given by the induced action of $\D(\QCoh Y)$ on $\K_\mathrm{ac}(\Inj X)$ as in \cite{Stevensonclass}*{Theorem 8.8}; in particular it gives rise to the support theory on $\D_\mathrm{Sg}(X)$ classifying thick subcategories.

\begin{thm}\label{thm_duality}
For every object $M$ of $\D^b(\Coh X)$ there is an equality
\begin{displaymath}
\supp_{\D(\QCoh Y)} \pi M = \supp_{\D(\QCoh Y)} \pi M^\vee.
\end{displaymath}
Hence every thick subcategory of $\D_\mathrm{Sg}(X)$ is self dual under the duality of $\D_\mathrm{Sg}(X)$ induced by Grothendieck duality.
\end{thm}
\begin{proof}
Let $M$ be an object of $\D^b(\Coh X)$ as in the statement. By Lemmas \ref{lem_pull} and \ref{lem_push} we have an isomorphism
\begin{displaymath}
(i_*p^*M)^\vee \iso i_*(p^*(M^\vee) \otimes i^!\str_Y).
\end{displaymath}
Using our definition of the support (or by \cite{StevensonActions}*{Remark~8.7}) the support is computed stalkwise. Thus as $i^!\str_Y$ is a line bundle and, by (essentially) \cite{Stevensonclass}*{Lemma~10.2}, $\Sing Y \cie i(\Sing Z)$ (so the pushforward doesn't affect the stalks we care about) we deduce the first equality below
\begin{displaymath}
\supp \pi (i_*p^*M)^\vee = \supp \pi i_*p^*(M^\vee) = \supp_{\D(\QCoh Y)} \pi M^\vee.
\end{displaymath}
On the other hand using the last lemma we can write
\begin{displaymath}
\supp \pi (i_*p^* M)^\vee = \supp \pi i_*p^*M = \supp_{\D(\QCoh Y)} \pi M.
\end{displaymath}
Combining these two sets of equalities we see $\pi M$ and $\pi M^\vee$ have identical supports.

By the classification of thick subcategories given in \cite{Stevensonclass}*{Theorem~8.8} it follows that a thick subcategory of $\D_\mathrm{Sg}(X)$ contains $\pi M$ if and only if it contains $\pi (M^\vee)$. Hence every thick subcategory of the singularity category of $X$ is stable under the duality given by $(-)^\vee$.

\end{proof}

When $X$ is the spectrum of a complete intersection ring this gives a different proof of \cite{AvBuchSupp}*{Theorem 5.6 (10)} in the case of maximal Cohen-Macaulay modules (as well as extending it to our finer notion of support).

For rings we can, of course, restate this in the language of maximal Cohen-Macaulay (MCM) modules and feel it is worth explicitly doing so. Let $R$ be a complete intersection ring and, as usual, denote by $\MCM R$ the Frobenius exact category of maximal Cohen-Macaulay $R$-modules and by $\sMCM R$ its stable category. Note that for a MCM module $M$, considered as an object of $\D^b(\modu R)$, there is an isomorphism $M^\vee = \Hom_R(M,R)$ i.e., Grothendieck duality just sends $M$ to the dual MCM module (recall that $\Hom_R(M,R)$ is also MCM).

In order to state the corollary let us recall the notion of a thick subcategory of an exact category.

\begin{defn}
Let $\mathcal{A}$ be an exact category. We say a full subcategory $\mathcal{B}\cie \mathcal{A}$ is \emph{thick} if $\mathcal{B}$ is closed under direct summands and if for every short exact sequence
\begin{displaymath}
0 \to A \to B \to C \to 0
\end{displaymath}
where two of $A, B$, and $C$ lie in $\mathcal{B}$ then the third object also belongs to $\mathcal{B}$.
\end{defn}

\begin{cor}\label{cor_dualityMCM}
Let $R$ be a complete intersection. Then every thick subcategory of $\sMCM R$ is closed under taking $R$-duals i.e., a thick subcategory of $\sMCM R$ contains $M$ if and only if it contains $\Hom_R(M,R)$. Moreover, this is already true in the exact category $\MCM R$: every thick subcategory of $\MCM R$ containing $R$ is closed under $R$-duals.
\end{cor}
\begin{proof}
The first statement follows from the equivalence of $\D_{\mathrm{Sg}}(R)$ and $\sMCM R$ (see \cite{Buchweitzunpub}) together with the observation that duality for the singularity category corresponds to taking $R$-duals under this equivalence. The statement for the exact category $\MCM R$ follows from the statement for the stable category via the bijection between the thick subcategories of $\MCM R$ containing $R$ and the thick subcategories of $\sMCM R$.
\end{proof}

\begin{rem}
Our results are also valid for rings which are locally abstract hypersurfaces.
\end{rem}

\begin{rem}
This answers a question posed in \cite{DTradius}*{Question~5.3}, namely it settles the implication $(1) \rimp (2)$. The question there concerns resolving subcategories but these are the same as thick subcategories containing $R$ by \cite{DTradius}*{Corollary~4.16}.
\end{rem}

\section{Symmetry of orthogonality}

Our result on duality has implications for the symmetry of the vanishing of homs in the singularity category. We continue to use the notation and hypotheses of the last section. Given a set of objects $\mathcal{A}\cie \D_\mathrm{Sg}(X)$ we denote the smallest thick subcategory containing $\mathcal{A}$ by $\langle \mathcal{A}\rangle$ and the right perpendicular of this thick subcategory by $\langle\mathcal{A}\rangle^\perp$. Note that the right perpendicular is again a thick subcategory.

\begin{thm}\label{thm_symm1}
Let $M$ and $N$ be objects of $\D_\mathrm{Sg}(X)$. Then 
\begin{displaymath}
N \in \langle M \rangle^\perp \;\; \text{if and only if} \;\; M \in \langle N \rangle^\perp.
\end{displaymath}
\end{thm}
\begin{proof}
First note that $N \in \langle M \rangle^\perp$ if and only if $\langle N \rangle \cie \langle M \rangle^\perp$. For $N' \in \langle N \rangle$ and $M' \in \langle M \rangle$ we have
\begin{align*}
\Hom_{\D_\mathrm{Sg}(X)}(N',M') &= \Hom_{\D_\mathrm{Sg}(X)^\op}(M',N') \\
&\iso \Hom_{\D_\mathrm{Sg}(X)}(\RHom(M',R),\RHom(N',R)).
\end{align*}
So, by Theorem~\ref{thm_duality}, we see that $\langle N \rangle\cie \langle M \rangle^\perp$ iff for all $N' \in \langle N \rangle$ and $M' \in \langle M \rangle$
\begin{displaymath}
\Hom_{\D_\mathrm{Sg}(X)}(N',M') \iso \Hom_{\D_\mathrm{Sg}(X)}(\RHom(M',R),\RHom(N',R)) = 0
\end{displaymath}
i.e., $\langle M \rangle \cie \langle N \rangle^\perp$. Hence $N\in \langle M \rangle^\perp$ if and only if $M \in \langle N \rangle^\perp$ as claimed.
\end{proof}

This gives a version for the singularity category of the symmetry of vanishing part of a theorem of Avramov and Buchweitz \cite{AvBuchSupp}*{Theorem~III}. It is perhaps worth pointing out that, in constrast to the work of Avramov and Buchweitz, we do not rely on cohomology operators.

We have the following corollary which gives a statement closer in spirit to the theorem of Avramov and Buchweitz (and, in fact, restates and gives a new proof of part of it).

\begin{cor}
Let $R$ be a complete intersection ring and let $M$ and $N$ be maximal Cohen-Macaulay $R$-modules. Then the following are equivalent:
\begin{itemize}
\item[$(i)$] $\Ext^i(M,N) = 0 \;\; \forall \; i\gg 0$;
\item[$(ii)$] $\Ext^i(N,M) = 0 \;\; \forall \; i\gg 0$;
\item[$(iii)$] $N \in \langle M \rangle^\perp$ in $\sMCM(R)$;
\item[$(iv)$] $M \in \langle N \rangle^\perp$ in $\sMCM(R)$.
\end{itemize}
\end{cor}
\begin{proof}
Corollary~\ref{cor_dualityMCM} gives the equivalence of (iii) and (iv) and it is evident that (iii) and (iv) imply (i) and (ii) respectively. We prove that (i) implies (iii) (the proof that (ii) implies (iv) being identical). Consider the full subcategory
\begin{displaymath}
\mathcal{K} = \{L\in \sMCM(R) \; \vert \; \Hom(M, \S^i L) = 0 \;\; \forall i\geq 0\}
\end{displaymath}
of the stable category. One sees easily that the subcategory $\mathcal{K}$ is a preaisle i.e., it is closed under sums, summands, positive suspensions and extensions. Using \cite{DTradius}*{Corollary 4.16} (and noting that the result remains true for complete intersections of the form we consider) we see it must therefore be a thick subcategory - it is also closed under desuspensions. Hence $\mathcal{K}$ is actually $\langle M \rangle^\perp$. The claim now follows as by (i) some $\S^i N \in \mathcal{K}$ and hence $N\in \mathcal{K} = \langle M \rangle^\perp$.
\end{proof}

\section{Thick subcategories of the bounded derived category}

We wish to extend the self-duality result for thick subcategories of the singularity category to the bounded derived category of $X$, where $X$ is as in Section \ref{sec_duality1}. We do this by first extending the classification of thick subcategories of $\D_\mathrm{Sg}(X)$ of \cite{Stevensonclass} to $\D^b(\Coh X)$. We note that Iyengar has reported \cite{IyengarLCI} a classification of the thick subcategories of $\D^b(\modu R)$, for $R$ a locally complete intersection essentially of finite type over a field, in terms of the spectrum of the Hochschild cohomology ring. When both results apply our result gives an alternative description of the thick subcategories in terms of $\Spec R$ and the generic hypersurface $Y$; from this one can recover the spectrum of the Hochschild cohomology, see Proposition \ref{prop_HH}.

Let $\Th^\otimes(\D^\mathrm{perf}(X))$ denote the lattice of thick tensor ideals of $\D^\mathrm{perf}(X)$. Similarly we let $\Th^\otimes(\D^b(\Coh X))$ and $\Th(\D_\mathrm{Sg}(X))$ denote the lattice of thick subcategories of $\D^b(\Coh X)$ which are closed under tensoring with objects of $\D^\mathrm{perf}(X)$ i.e.\ the $\D^\mathrm{perf}(X)$\emph{-submodules}, and the lattice of thick subcategories of $\D_\mathrm{Sg}(X)$ respectively. For a collection of objects $\mathcal{M} \cie \D^b(\Coh X)$ we will denote by $\langle \mathcal{M} \rangle_\otimes$ the smallest $\D^\mathrm{perf}(X)$-submodule containing $\mathcal{M}$ and we use similar notation for ideals in the perfect complexes and thick subcategories of the singularity category.

Recall that there are lattice isomorphisms between $\Th^\otimes(\D^\mathrm{perf}(X))$ and the lattice of specialization closed subsets of $X$ by work of Thomason \cite{Thomclass} and between $\Th(\D_\mathrm{Sg}(X))$ and specialization closed subsets of $\Sing Y$ where $Y$ is a generic hypersurface as in \eqref{eq_diag}. We will use these classifications together with the quotient sequence
\begin{displaymath}
\D^\mathrm{perf}(X) \to \D^b(\Coh X) \to \D_\mathrm{Sg}(X)
\end{displaymath}
to describe $\Th^\otimes(\D^b(\Coh X))$. We note that this is not possible for a general quotient sequence. 

For $\mathcal{K} \in \Th^\otimes(\D^b(\Coh X))$ we set
\begin{displaymath}
\mathcal{K}^\mathrm{perf} = \mathcal{K} \intersec \D^\mathrm{perf}(X) \quad \text{and} \quad \mathcal{K}_\mathrm{Sg} = \mathcal{K}/\mathcal{K}^\mathrm{perf}.
\end{displaymath}
We associate to $\mathcal{K}$ a pair of subsets, namely $\supph \mathcal{K}$, the homological support of $\mathcal{K}$
\begin{displaymath}
\supph \mathcal{K} = \{x\in X \; \vert \; \mathcal{K}_x \neq 0\} = \bigcup_{E\in \mathcal{K}} \{x\in X \; \vert \; E_x\neq 0\},
\end{displaymath}
and $\supp_{\D(\QCoh Y)}\pi \mathcal{K}$, the singular support of $\mathcal{K}$
\begin{displaymath}
\supp_{\D(\QCoh Y)} \pi\mathcal{K} = \bigcup_{E\in \mathcal{K}} \supp_{\D(\QCoh Y)} \pi E.
\end{displaymath}

The first computation we need, concerning the perfect part of $\mathcal{K}$, follows easily from the following slight generalization of a result of Dwyer, Greenlees, and Iyengar. The proof of this extension uses the machinery of tensor actions \cite{StevensonActions}; we prefer not to get into the preliminaries required here as this machinery is not needed explicitly elsewhere.

\begin{lem}\label{lem_genDGI}
Let $M$ be an object of $\D^b(\Coh X)$. Then $\mathcal{M} = \langle M \rangle_\otimes$ contains a perfect complex $W$ such that $M$ is in the localizing tensor ideal of $\D(\QCoh X)$ generated by $W$ i.e., in (a slight corruption of) the language of \cite{DGI_finiteness} $M$ is $\otimes$-proxy small.
\end{lem}
\begin{proof}
Let $X = \cup_{i=1}^n U_i$ be an open affine cover of $X$ such that each $U_i$ is the spectrum of a complete intersection; such a cover exists as $X$ is quasi-compact and the zero locus of a regular section. For each $i$ there exists, by \cite{DGI_finiteness}*{Theorem~9.4} (noting that the local hypothesis is not necessary in the case we are interested in), in $\D^b(\Coh U_i)$ a $W_i$ in $\langle M\vert_{U_i} \rangle$ as in the statement. We can obtain from each $W_i$ a $\widetilde{W}_i \in \D^\mathrm{perf}(X)$, such that
\begin{displaymath}
\supph_{U_i} W_i \cie \supph \widetilde{W}_i \cie \supph M.
\end{displaymath}
This is done by lifting $W_i \oplus \S W_i$ and then, if necessary, tensoring with a perfect complex with the same support as $M$ in order to fix the support. Thus $(\widetilde{W}_j)\vert_{U_i} \in \langle M\vert_{U_i}\rangle$ for each $i$ and $j$ by the classification of thick subcategories of $\D^\mathrm{perf}(U_i)$ since 
\begin{displaymath}
\supph_{U_i} (\widetilde{W}_j)\vert_{U_i} \cie \supph_{U_i} M\vert_{U_i} = \supph_{U_i} W_i.
\end{displaymath}
We now use the local-to-global principle to deduce the desired statement. There is an action of $\D(\QCoh X)$ on $\K(\Inj X)$, denoted $\odot$, which restricts to an action of $\D^\mathrm{perf}(X)$ on $\K(\Inj X)^c \iso \D^b(\Coh X)$ which is compatible with this latter equivalence (\cite{Stevensonclass}*{Lemma~5.4}). Since the structure we are interested in is compatible with this equivalence we don't introduce notation to distinguish between objects of $\D(\QCoh X)$ and their K-injective resolutions viewed as objects of $\K(\Inj X)$.

Let $\mathcal{N}$ be the smallest localizing subcategory of $\K(\Inj X)$ containing $M$ and closed under the $\D(\QCoh X)$-action. Since $\mathcal{N}$ contains each $M\vert_{U_i}$ it contains the $(\widetilde{W}_j)\vert_{U_i}$ and thus by the local-to-global principle (see \cite{StevensonActions}*{6.1 and 6.9}) it contains the $\widetilde{W}_j$. Now observe that there are equalities
\begin{align*}
\mathcal{N}^c &= \langle E\odot M \; \vert \; E\in \D^\mathrm{perf}(X) \rangle_\mathrm{loc} \intersec \K(\Inj X)^c \\
&= \langle E \odot M \; \vert \; E\in \D^\mathrm{perf}(X) \rangle \\
&\iso \mathcal{M}
\end{align*}
where the first equality and last equivalence use \cite{StevensonActions}*{Lemma~3.10}. Thus the $\widetilde{W}_j$ are all in $\mathcal{M}$ and it follows from the construction together with the classification of localizing ideals of $\D(\QCoh X)$ (see for instance \cite{AJS3}*{Corollary~4.13}) that the localizing ideal generated by $\oplus_i \widetilde{W}_i$ contains $M$.
\end{proof}

\begin{lem}\label{lem_perfsupp}
Let $\mathcal{K}$ be a thick $\D^\mathrm{perf}(X)$-submodule of $\D^b(\Coh X)$. The thick subcategory $\mathcal{K}^\mathrm{perf}$ is non-zero if $\mathcal{K}$ is non-zero and there is an equality
\begin{displaymath}
\supph \mathcal{K} = \supph \mathcal{K}^\mathrm{perf}.
\end{displaymath}
\end{lem}
\begin{proof}
For every object $M$ of $\D^b(\Coh X)$ there is, by the last lemma, a perfect complex $W$ with
\begin{displaymath}
W \in \langle M \rangle_\otimes \quad \text{and} \quad M\in \langle W \rangle_{\mathrm{loc},\otimes}.
\end{displaymath}
It follows in a straightforward way from the properties of the support and \cite{AJS3}*{Corollary~4.13} that $\supph W = \supph M$. Both statements of the lemma are thus immediate.
\end{proof}

\begin{lem}\label{lem_sg_ff}
Let $\mathcal{K}$ be a thick $\D^\mathrm{perf}(X)$-submodule of $\D^b(\Coh X)$. Then the natural functor $\mathcal{K}_\mathrm{Sg} \to \D_\mathrm{Sg}(X)$ is fully faithful.
\end{lem}
\begin{proof}
It is sufficient to show that any map from a perfect complex to an object of $\mathcal{K}$ factors through an object of $\mathcal{K}^\mathrm{perf}$.

By the last lemma there is an equality
\begin{displaymath}
\mathcal{K}^\mathrm{perf} = \D^\mathrm{perf}_{\supph \mathcal{K}}(X)
\end{displaymath}
of thick subcategories, where the category on the right is, as usual, the full subcategory of perfect complexes supported on $\supph \mathcal{K}$. Suppose we have $P\in \D^\mathrm{perf}(X)$ and $E\in \mathcal{K}$ so, in particular, $E\in \D^b_{\supph E}(X)$. By \cite{Orlov_formal}*{Lemma 2.6} any $P\to E$ factors via an object of $\D^\mathrm{perf}_{\supph E}(X) \cie \mathcal{K}^\mathrm{perf}$ as required which proves the lemma.
\end{proof}

We now define the lattice which will control the thick $\D^\mathrm{perf}(X)$-submodules of $\D^b(\Coh X)$.
\begin{defn}
We use the notation of diagram \eqref{eq_diag}. We set 
\begin{displaymath}
\mathcal{S} = \{(\mathcal{V},\mathcal{W}) \in X \times \Sing Y \; \vert \; \mathcal{V},\mathcal{W} \; \text{are specialization closed, and} \; pi^{-1}(\mathcal{W}) \cie \mathcal{V}\}.
\end{displaymath}
Componentwise inclusions, intersections, and unions give $\mathcal{S}$ a lattice structure.
\end{defn}

\begin{lem}\label{lem_sigma}
There is a function $\sigma \colon \Th^\otimes(\D^b(\Coh X)) \to \mathcal{S}$ given by
\begin{displaymath}
\s\mathcal{K} = (\supph_X \mathcal{K}, \supp_{\D(\QCoh Y)} \mathcal{K}_\mathrm{Sg}),
\end{displaymath}
where we view $\mathcal{K}_\mathrm{Sg}$ as a full subcategory of $\D_\mathrm{Sg}(X)$ using Lemma \ref{lem_sg_ff}.
\end{lem}
\begin{proof}
All we need to check is that
\begin{displaymath}
pi^{-1} \supp_{\D(\QCoh Y)} \mathcal{K}_\mathrm{Sg} \cie \supph_X \mathcal{K}.
\end{displaymath}
This is easily seen from the description of the singular support given before Lemma \ref{lem_dualsupp}: if $y\in \supp_{\D(\QCoh Y)} \pi M$ i.e., in $\supp_{\D(\QCoh Y)} \pi i_*p^* M$ then certainly $y\in \supph_{Z} p^*M$ and so $p(y) \in \supph_X M$.
\end{proof}

\begin{lem}\label{lem_tau}
There is a function $\tau \colon \mathcal{S} \to \Th^\otimes(\D^b(\Coh X))$ given by
\begin{displaymath}
\tau(\mathcal{V},\mathcal{W}) = \{E\in \D^b(\Coh X) \; \vert \; \supph_X E \cie \mathcal{V} \; \text{and} \; \supp_{\D(\QCoh Y)}\pi E \cie \mathcal{W}\}.
\end{displaymath}
\end{lem}
\begin{proof}
All one needs to check is that $\tau(\mathcal{V},\mathcal{W})$ is a thick $\D^\mathrm{perf}(X)$-submodule and this is easily deduced from the properties of the support (see \cite{StevensonActions}*{Proposition 5.7} for properties of $\supp_{D(\QCoh Y)}$).
\end{proof}

We now prove the main result of this section, namely that $\t$ and $\s$ give a bijection. 

\begin{lem}\label{lem_bij1}
The map $\tau$ is a split monomorphism with left inverse $\s$ i.e., $\s\t = \id_\mathcal{S}$.
\end{lem}
\begin{proof}
Say $(\mathcal{V},\mathcal{W}) \in \mathcal{S}$. We need to check that
\begin{displaymath}
\s\t(\mathcal{V},\mathcal{W}) = (\supph \t(\mathcal{V},\mathcal{W}), \supp_{\D(\QCoh Y)} \t(\mathcal{V},\mathcal{W})_\mathrm{Sg})
\end{displaymath}
is equal to $(\mathcal{V},\mathcal{W})$.

First observe that
\begin{align*}
\supph \t(\mathcal{V},\mathcal{W}) &= \supph (\t(\mathcal{V},\mathcal{W}) \intersec \D^\mathrm{perf}(X)) \\
&= \supph \{E\in \D^\mathrm{perf}(X) \; \vert \; \supph E \cie \mathcal{V} \} \\
&= \mathcal{V}
\end{align*}
where the first equality is Lemma~\ref{lem_perfsupp} and the last equality follows (for example) from the classification of thick ideals of $\D^\mathrm{perf}(X)$ given in \cite{Thomclass}.

Now consider $\supp_{\D(\QCoh Y)} \t(\mathcal{V},\mathcal{W})_\mathrm{Sg}$. It is clear that this subset is contained in $\mathcal{W}$, so it remains to show it is all of $\mathcal{W}$. By Lemma \ref{lem_sg_ff} there is a fully faithful inclusion of $\t(\mathcal{V},\mathcal{W})_\mathrm{Sg}$ into $\D_\mathrm{Sg}(X)$ and we shall denote the image of this inclusion by $\mathcal{M}$. By the definition of $\t$ the idempotent completion of $\mathcal{M}$ in $\D_\mathrm{Sg}(X)$ is 
\begin{displaymath}
\{A \in \D_\mathrm{Sg}(X) \; \vert \; \supp_{\D(\QCoh Y)} A \cie \mathcal{W}\}.
\end{displaymath}
By the classification of \cite{Stevensonclass}*{Theorem 8.8} the singular support of this subcategory is precisely $\mathcal{W}$ and so this must already be true for $\mathcal{M}$. Thus $\supp_{\D(\QCoh Y)} \t(\mathcal{V},\mathcal{W})_\mathrm{Sg} = \mathcal{W}$ completing the proof.
\end{proof}

\begin{lem}\label{lem_bij2}
The map $\s$ is a split monomorphism with left inverse $\t$ i.e., $\t\s = \id_\mathcal{S}$.
\end{lem}
\begin{proof}
Suppose $\mathcal{K} \in \Th^\otimes(\D^b(\Coh X))$ and consider
\begin{align*}
\t\s\mathcal{K} &= \tau(\supph \mathcal{K}, \supp_{\D(\QCoh Y)} \mathcal{K}_\mathrm{Sg}) \\
&= \{ E\in \D^b(\Coh X) \; \vert \; \supph E \cie \supph \mathcal{K},\; \supp_{\D(\QCoh Y)} \pi E \cie \supp_{\D(\QCoh Y)} \mathcal{K}_\mathrm{Sg} \}.
\end{align*}
We have $\mathcal{K} \cie \t\s\mathcal{K}$ and we need to check the reverse inclusion.

First observe that
\begin{displaymath}
(\t\s\mathcal{K})^\mathrm{perf} = \{E\in \D^\mathrm{perf}(X) \; \vert \; \supph E \cie \supph \mathcal{K}^\mathrm{perf}\} = \mathcal{K}^\mathrm{perf}
\end{displaymath}
by Lemma \ref{lem_perfsupp} and the classification result for $\D^\mathrm{perf}(X)$. Thus we have fully faithful functors
\begin{displaymath}
\mathcal{K}_\mathrm{Sg} \to (\t\s\mathcal{K})_\mathrm{Sg} \to \D_\mathrm{Sg}(X).
\end{displaymath}
The classification of thick subcategories for $\D_\mathrm{Sg}(X)$ implies that the idempotent completions of $\mathcal{K}_\mathrm{Sg}$ and $(\t\s\mathcal{K})_\mathrm{Sg}$ in $\D_\mathrm{Sg}(X)$ agree as they correspond to the same subset of $\Sing Y$. As $\mathcal{K}_\mathrm{Sg}$ is already idempotent complete in $(\t\s\mathcal{K})_\mathrm{Sg}$ these two categories must agree. Thus
\begin{align*}
0 = (\t\s\mathcal{K})_\mathrm{Sg} / \mathcal{K}_\mathrm{Sg} &= (\t\s\mathcal{K} / \mathcal{K}^\mathrm{perf}) / (\mathcal{K} / \mathcal{K}^\mathrm{perf}) \\
&\iso \t\s\mathcal{K} / \mathcal{K}
\end{align*}
showing that $\t\s\mathcal{K} = \mathcal{K}$ as desired.
\end{proof}

Combining the last two lemmas proves the following classification theorem.

\begin{thm}\label{thm_bij}
The assignments $\s$ and $\t$ give an inclusion preserving bijection between $\Th^\otimes(\D^b(\Coh X))$ and 
\begin{displaymath}
\mathcal{S} = \{(\mathcal{V},\mathcal{W}) \in X \times \Sing Y \; \vert \; \mathcal{V},\mathcal{W} \; \text{are specialization closed, and} \; pi^{-1}(\mathcal{W}) \cie \mathcal{V}\}.
\end{displaymath}
In particular, if $X\iso \Spec R$ is affine then this gives a classification of thick subcategories of $\D^b(\modu R)$.
\end{thm}

\begin{rem}
It is appropriate at this juncture to make some remarks concerning special cases of the bijection and similar results in the literature. As we have already mentioned Iyengar has a proof of this classification (\cite{IyengarLCI}) for locally complete intersections essentially of finite type over a field; this particular result will be discussed further below.

In \cite{TakahashiLPS}*{Theorem 3.13} Takahashi gives a partial proof of the theorem for abstract hypersurface rings. Namely, he considers the case of thick subcategories containing the perfect complexes i.e., elements of the form $(\Spec R, \mathcal{W}) \in \mathcal{S}$. This reduces to the classification for thick subcategories of the singularity category (at the other extreme, where one considers pairs $(\mathcal{V},\varnothing)$ one reduces to the classification of thick ideals of perfect complexes).

The full classification for the bounded derived categories of elementary abelian $p$-groups goes back to the work of Benson, Carlson, and Rickard \cite{BCR}; the extension from the singularity category to the bounded derived category is simplified in this case as the rings involved are artinian.
\end{rem}

We have the following corollary of the classification theorem extending our self-duality result for the singularity category.

\begin{cor}\label{cor_duality}
Every thick $\D^\mathrm{perf}(X)$-submodule of $\D^b(\Coh X)$ is dual to itself under Grothendieck duality i.e., if $\mathcal{K} \in \Th^\otimes(\D^b(\Coh X))$ then $\RsHom(-,\str_X)$ gives an equivalence $\mathcal{K} \stackrel{\sim}{\to} \mathcal{K}^\op$.
\end{cor}
\begin{proof}
By the theorem a thick submodule $\mathcal{K}$ of $\D^b(\Coh X)$ is determined by $\supph \mathcal{K}^\mathrm{perf}$ (this also uses Lemma \ref{lem_perfsupp}) and $\supp_{\D(\QCoh Y)}\mathcal{K}_\mathrm{Sg}$. Taking duals in either $\D^\mathrm{perf}(X)$ or $\D_\mathrm{Sg}(X)$ does not change the relevant supports, by rigidity of the former category and Theorem \ref{thm_duality} for the latter. Thus, as the quotient sequence giving rise to $\D_\mathrm{Sg}(X)$ is compatible with the duality, the image of $\mathcal{K}$ under $\RsHom(-,\str_X)$ is determined by the same support theoretic data as $\mathcal{K}$ and so is just the thick submodule $\mathcal{K}^\op$ of $\D^b(R)^\op$.
\end{proof}

This leads us to the rather natural question, also posed in \cite{DTradius}*{Question~5.3}. Is the converse of Corollary \ref{cor_duality} true? Explicitly, if $X$ is a noetherian separated Gorenstein scheme such that every thick submodule of $\D^b(\Coh X)$ is self-dual under $\RsHom(-,\str_X)$ is $X$ a complete intersection (possibly in a more general sense than we have used i.e., locally a complete intersection)? I am aware of little evidence that this should be the case, so, restricting to the affine case, it is perhaps better to ask what class of rings this condition classifies? 

Of course it would also be interesting to understand how the thick subcategories of $\D^b(\modu R)$ for a general Gorenstein $R$ (or more generally scheme $X$) are permuted by the duality. However, this seems at the current time to be out of reach; no general classification of thick subcategories has been obtained nor does there even seem to be any reasonable candidates for a support datum.

As in the case of the singularity category the duality result we have proved has consequences for vanishing of cohomology. We restrict ourselves to the case $X = \Spec R$ is affine. As the proof is basically identical to the proof of Theorem \ref{thm_symm1} we do not include it.

\begin{cor}\label{cor_dbduality}
Let $M$ and $N$ be objects of $\D^b(\modu R)$. Then 
\begin{displaymath}
N \in \langle M \rangle^\perp \;\; \text{if and only if} \;\; M \in \langle N \rangle^\perp.
\end{displaymath}
\end{cor}

This extends part of a theorem of Avramov and Buchweitz \cite{AvBuchSupp}*{Theorem III} to the bounded derived category.

We end with an amusing consequence of the classification: one can, at least in theory, use the classification to compute the spectrum of some Hochschild cohomology rings via Iyengar's result \cite{IyengarLCI}.

\begin{prop}\label{prop_HH}
Let $R$ be a complete intersection that is essentially of finite type over a field $K$. Denote by $cl(\mathcal{S})$ the sublattice of $\mathcal{S}$ consisting of pairs $(\mathcal{V},\mathcal{W})$ where $\mathcal{V}$ and $\mathcal{W}$ are closed in $\Spec R$ and $\Sing Y$ respectively. Then $\Spec cl(\mathcal{S})$, the spectral space of prime filters on $cl(\mathcal{S})$, is naturally homeomorphic to $\Spech HH^*(R/K)$ the homogeneous spectrum of the Hochschild cohomology of $R$ over $K$.
\end{prop}
\begin{proof}
There are lattice isomorphisms
\begin{displaymath}
\left\{ \begin{array}{c}
\text{specialization closed subsets} \\ \text{of}\; \Spech HH^*(R/K) 
\end{array} \right\}
\xymatrix{ \ar[r]<1ex>^{t} \ar@{<-}[r]<-1ex>_{s} &} \left\{
\begin{array}{c}
\text{thick subcategories of} \; \D^b(R) 
\end{array} \right\}
\xymatrix{ \ar[r]<1ex>^{\s} \ar@{<-}[r]<-1ex>_{\tau} &} \mathcal{S}
\end{displaymath}
where the first isomorphism is due to Iyengar \cite{IyengarLCI}, $s$ is as in \cite{BIKstrat2}*{Theorem 6.1}, and $t$ is the obvious inverse. One checks easily that closed subsets of $\Spech HH^*(R/K)$ correspond to singly generated thick subcategories of $\D^b(R)$ which in turn correspond to elements of $cl(\mathcal{S})$. Now one just needs to note that, by the theory of duality for distributive lattices, $\Spech HH^*(R/K)$ is determined by its lattice of closed subsets; by the above bijections and observation this lattice is isomorphic to $cl(\mathcal{S})$.
\end{proof}

\begin{ack}
I am grateful to Hailong Dao and Ryo Takahashi for asking me the question, concerning closure of thick subcategories of maximal Cohen-Macaulay modules under duals, which motivated this work. It is also a pleasure to thank Srikanth Iyengar for interesting discussions on this subject and Hailong Dao for helpful comments.
\end{ack}

  \bibliography{greg_bib}

\begin{bibdiv}
\begin{biblist}

\bib{AvBuchSupp}{article}{
      author={Avramov, Luchezar~L.},
      author={Buchweitz, Ragnar-Olaf},
       title={Support varieties and cohomology over complete intersections},
        date={2000},
        ISSN={0020-9910},
     journal={Invent. Math.},
      volume={142},
      number={2},
       pages={285\ndash 318},
      review={\MR{MR1794064 (2001j:13017)}},
}

\bib{AJS3}{article}{
      author={{Alonso Tarr{\'{\i}}o}, L.},
      author={{Jerem{\'{\i}}as L{\'o}pez}, A.},
      author={{Souto Salorio}, M.~J.},
       title={Bousfield localization on formal schemes},
        date={2004},
     journal={J. Algebra},
      volume={278},
      number={2},
       pages={585\ndash 610},
}

\bib{BCR}{article}{
      author={Benson, D.~J.},
      author={Carlson, Jon~F.},
      author={Rickard, Jeremy},
       title={Thick subcategories of the stable module category},
        date={1997},
        ISSN={0016-2736},
     journal={Fund. Math.},
      volume={153},
      number={1},
       pages={59\ndash 80},
      review={\MR{MR1450996 (98g:20021)}},
}

\bib{BIKstrat2}{article}{
      author={Benson, Dave},
      author={Iyengar, Srikanth~B.},
      author={Krause, Henning},
       title={Stratifying triangulated categories},
        date={2011},
     journal={J. Topol.},
      volume={4},
      number={3},
       pages={641\ndash 666},
}

\bib{Buchweitzunpub}{article}{
      author={Buchweitz, Ragnar-Olaf},
       title={Maximal {C}ohen-{M}acaulay modules and {T}ate cohomology over
  {G}orenstein rings},
        date={1987},
     journal={unpublished},
}

\bib{DGI_finiteness}{article}{
      author={Dwyer, W.},
      author={Greenlees, J. P.~C.},
      author={Iyengar, S.},
       title={Finiteness in derived categories of local rings},
        date={2006},
     journal={Comment. Math. Helv.},
      volume={81},
      number={2},
       pages={383\ndash 432},
}

\bib{DTradius}{unpublished}{
      author={Dao, H.},
      author={Takahashi, R.},
       title={The radius of a subcategory of modules},
        date={2011},
        note={arXiv:1111.2902v2},
}

\bib{HartshorneRD}{book}{
      author={Hartshorne, Robin},
       title={Residues and duality},
      series={Lecture notes of a seminar on the work of A. Grothendieck, given
  at Harvard 1963/64. With an appendix by P. Deligne. Lecture Notes in
  Mathematics, No. 20},
   publisher={Springer-Verlag},
     address={Berlin},
        date={1966},
      review={\MR{MR0222093 (36 \#5145)}},
}

\bib{IyengarLCI}{article}{
      author={Iyengar, Srikanth},
       title={Stratifying derived categories associated to finite groups and
  commutative rings},
     journal={Kyoto RIMS Workshop on Algebraic Triangulated Categories and
  Related Topics},
  volume={http://www.mi.s.osakafu-u.ac.jp/~kiriko/seminar/09JulRIMS/lectures/iyengar.pdf},
}

\bib{KrStab}{article}{
      author={Krause, Henning},
       title={The stable derived category of a {N}oetherian scheme},
        date={2005},
        ISSN={0010-437X},
     journal={Compos. Math.},
      volume={141},
      number={5},
       pages={1128\ndash 1162},
      review={\MR{MR2157133 (2006e:18019)}},
}

\bib{MatsuAlg}{book}{
      author={Matsumura, Hideyuki},
       title={Commutative algebra},
     edition={Second},
      series={Mathematics Lecture Note Series},
   publisher={Benjamin/Cummings Publishing Co., Inc., Reading, Mass.},
        date={1980},
      volume={56},
        ISBN={0-8053-7026-9},
      review={\MR{MR575344 (82i:13003)}},
}

\bib{OrlovSing2}{article}{
      author={Orlov, D.~O.},
       title={Triangulated categories of singularities, and equivalences
  between {L}andau-{G}inzburg models},
        date={2006},
        ISSN={0368-8666},
     journal={Mat. Sb.},
      volume={197},
      number={12},
       pages={117\ndash 132},
         url={http://dx.doi.org/10.1070/SM2006v197n12ABEH003824},
      review={\MR{MR2437083 (2009g:14013)}},
}

\bib{Orlov_formal}{article}{
      author={Orlov, Dmitri},
       title={Formal completions and idempotent completions of triangulated
  categories of singularities},
        date={2011},
     journal={Adv. Math.},
      volume={226},
      number={1},
       pages={206\ndash 217},
}

\bib{Stevensonclass}{unpublished}{
      author={Stevenson, G.},
       title={Subcategories of singularity categories via tensor actions},
        date={2012},
        note={arXiv:1105.4698},
}

\bib{StevensonActions}{article}{
      author={Stevenson, G.},
       title={Support theory via actions of tensor triangulated categories},
     journal={J. Reine Angew. Math.},
        note={to appear},
}

\bib{TakahashiLPS}{article}{
      author={Takahashi, Ryo},
       title={Thick subcategories over {G}orenstein local rings that are
  locally hypersurfaces on the punctured spectra},
     journal={J. Math. Soc. Japan},
        note={to appear},
}

\bib{Thomclass}{article}{
      author={Thomason, R.~W.},
       title={The classification of triangulated subcategories},
        date={1997},
        ISSN={0010-437X},
     journal={Compositio Math.},
      volume={105},
      number={1},
       pages={1\ndash 27},
         url={http://dx.doi.org/10.1023/A:1017932514274},
      review={\MR{MR1436741 (98b:18017)}},
}

\end{biblist}
\end{bibdiv}

\end{document}